\documentclass[12pt]{article}
\usepackage{tikz}
\usepackage{microtype}
\usepackage{geometry}                
\usepackage{amssymb,amsmath,amsthm}
\usepackage{ucs}
\usepackage{color}
\usepackage{hyperref}
\usepackage{wrapfig}
\usepackage{graphicx}
\usepackage{url}
\usepackage{tikz}

\usepackage{ifxetex}
\ifxetex
\usepackage{fontspec,xltxtra,xunicode}
\usepackage{unicode-math}
\defaultfontfeatures{Mapping=tex-text}
\newfontfamily\cyrillicfont{STIX Two Text}
\setromanfont[Mapping=tex-text]{STIX Two Text}
\else
\usepackage{mathptmx}
\fi

\newtheorem*{theorem}{Theorem}
\newtheorem*{claim}{Claim}
\newtheorem{lemma}{Lemma}
\theoremstyle{remark}
\newtheorem{definition}{Definition}



\let\eps=\varepsilon
\let\le=\leqslant

\begin{document}
\title{Every computable set is generically reducible to every computable set that does not have density $0$ or $1$}
\author{Ruslan Ishkuvatov}
\maketitle

\begin{abstract}

The notion of generic reducibility was introduced by A.~Rybalov in his CiE 2018 paper~\cite{rybalov_cie}: a set $A$ is generically reducible to set $B$ if there exists a total computable function $f$ that $m$-reduces $A$ to $B$ such that the $f$-preimage of every set that has density $0$ has density $0$. It may be considered as the ``generic version'' of the notion of $m$-reducibility.

In this note we improve one of his results \cite[Theorem 2]{rybalov_cie}  and show that every two computable sets that do not have density $0$ or $1$ are equivalent with respect to generic reducibility, and  that every computable set is reducible to every computable set that does not have density $0$ or $1$, thus providing a complete classification of computable sets with respect to generic reducibility.

\end{abstract}

\section{Definitions and results}

\begin{definition}
For a set $A \subset \mathbb{N}$ we define the \emph{density} $\rho_n(A)$ as the fraction of $A$'s elements among the first $n$ natural numbers:
\[
\rho_n(A) = \frac{\#\{k<n\mid k\in A\}}{n}
\]
The limit  $\lim_n\rho_n(A)$ (if it exists) is called the \emph{asymptotic density} of $A$ and denoted by $\rho(A)$.
\end{definition}

\begin{definition}
A set $A\subset\mathbb{N}$ is \emph{negligible} if $\rho(A) = 0$ and \emph{generic} if $\rho(A) = 1$.
\end{definition}

\begin{definition}
A total function $f\colon \mathbb{N}\to \mathbb{N}$ is called \emph{uniform} if $f^{-1}(S)$ is negligible for every negligible set $S \subset \mathbb{N}$.
\end{definition}

\begin{definition}
We say that a set $A\subset\mathbb{N}$ is \emph{generically m-reducible} to a set $B\subset\mathbb{N}$ (and write $A \leq_{gm} B$) if there exists a uniform computable function $f:\mathbb{N} \to \mathbb{N}$ that $m$-reduces $A$ to $B$, i.e., $x \in A \iff f(x) \in B$ for all $x\in\mathbb{N}$.
\end{definition}

The notion of generic m-reducibility is a special case of m-reducibility: if $A \leq_{gm} B$ then $A \leq_{m} B$. The reverse implication is not true: if $A \leq_{gm} B$ and $B$ is negligible, then $A$ is also negligible (being the preimage of the negligible set $B$ under the uniform reduction function). 

In \cite{rybalov_cie} all computable sets \emph{that have limit density} were classified with respect to the generic reducibility. Namely, Theorem 2 (p.~361) says (for computable sets $A,B\subset\mathbb{N}$ that are not empty and have non-empty complement, and have some limit density):
\begin{enumerate}
 \item if $\rho(A) = \rho(B) = 1$, then $A \leq_{gm} B$.
 \item if $\rho(A) = \rho(B) = 0$, then $A \leq_{gm} B$.
 \item if $\rho(A),\rho(B) \ne 0$ and $\rho(A),\rho(B) \ne 1$, then $A \leq_{gm} B$.
\end{enumerate}

We show that the condition of having limit density may be omitted:

\begin{theorem}
Every two computable sets that do not have density $0$ or $1$ are equivalent with respect to generic reducibility. Moreover,  every computable set is generically reducible to every computable set that does not have density $0$ or $1$.
\end{theorem}

In this way we get a complete classification of computable sets with respect to generic reducibility: 
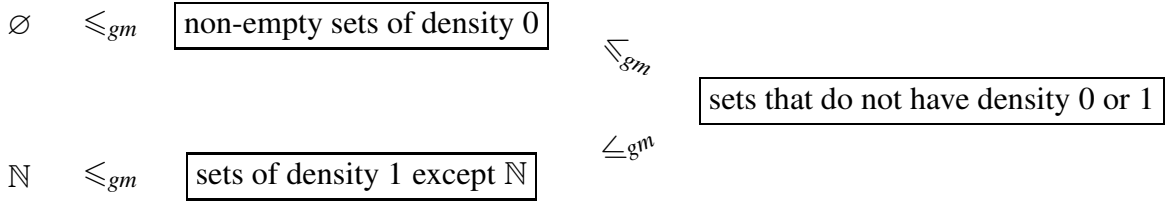
\begin{center}
\begin{tikzpicture}
\draw(0,2) node{$\varnothing$};
\draw (1.2,2) node{$\le_{gm}$};
\draw(0,0) node{$\mathbb{N}$};
\draw (1.2,0) node{$\le_{gm}$};
\draw (4.5,2) node{\fbox{non-empty sets of density $0$}};
\draw (4.5,0) node{\fbox{sets of density $1$ except $\mathbb{N}$}};
\draw (12,1) node{\fbox{sets that do not have density $0$ or $1$}};
\draw (8,1.6) node[rotate=-25]{$\le_{gm}$};
\draw (8,0.4) node[rotate=25]{$\le_{gm}$};
\end{tikzpicture}
\end{center}
(all sets are assumed to be computable).

\section{Proofs}

We start with a result that does not mention computability (and then consider its effective version):

\begin{theorem}
Let $A$ be an arbitrary subset of $\mathbb{N}$. Assume that $B$ is a subset of $\mathbb{N}$ and neither $B$ nor its complement are negligible. Then there exists a uniform \textup(total\textup) function that reduces $A$ to $B$, i.e., a function that maps $A$ into $B$ and the complement of $A$ into the complement of $B$.
\end{theorem}

\begin{proof}
The statement of the theorem follows from two lemmas.

\begin{lemma}\label{1}
For an arbitrary set $A$ there exists a uniform function that reduces $A$ to the set of even numbers \textup(i.e., each element of $A$ is mapped to an even number and each element of its complement is mapped to an odd number\textup).
\end{lemma}

\begin{lemma}\label{2}
Let $B\subset\mathbb{N}$ be a set such that neither $B$ nor its complement is negligible. Then there exists a uniform function that reduces the set of even numbers to $B$.
\end{lemma}

Indeed, the composition of these two functions reduces $A$ to $B$ while being uniform.

\begin{proof}[Proof of Lemma~\ref{1}]
We map elements of $A$ to different even numbers preserving the order: the least element of $A$ is mapped to $0$, the next one is mapped to $2$ and so forth. In the same way the elements of $A$'s complement are mapped to odd numbers. (If $A$ or its complement are finite some elements of $\mathbb{N}$ are not in the range of the function.)

Let us show that this map (called $f$ in the sequel) is uniform. Let $B$ be a negligible set. Consider an arbitrary initial segment of $\mathbb{N}$. We are interested in the density of the preimage of $B$ in this segment. The segment can be divided in two parts: the elements of $A$ and the elements of its complement. The density of $f^{-1}(B)$ in the segment is a weighted average of its densities in the both parts. These two densities are the densities of $B$ in some initial segments of even numbers and odd numbers respectively. If both $A$ and its complement are infinite, the lengths of these segments grow infinitely, and since $B$ is negligible, both densities converge to $0$. If $A$ or its complement are finite, the weight of the finite part in the weighted average of densities converges to $0$, so this part can be ignored.
\end{proof}

\begin{proof}[Proof of Lemma~\ref{2}]
Let us prove an auxilary statement first:

\begin{claim}
 Assume that a set $B\subset\mathbb{N}$ is not negligible. Then there exists a uniform function $f$ whose range is contained in $B$.
\end{claim}


\begin{proof}[Proof of the claim]
Suppose that the density of $B$ in initial segments exceeds some $\eps$ infinitely often (this happens for some $\eps>0$ if $B$ is not negligible). Then we can split $\mathbb{N}$ into intervals such that in each interval the density of $B$ is greater than $\eps$ (we can construct these intervals consecutively; if the next interval is much longer than the all previous ones combined, the density of $B$ in this interval is close to its density in some initial segment and therefore exceeds $\eps$ at some moment). Let the lengths of these intervals be $n_1, n_2, \ldots$ respectively. Without loss of generality we may assume that $n_1 \ll n_2 \ll \ldots$ (see below about the exact requirements for the lengths).
The number of elements of $B$ in these intervals is at least $\eps n_1$, $\eps n_2$, etc.

To construct the function $f$ with the required properties, let us choose some $N_1 \ll N_2 \ll \ldots$ (much greater than $n_1, n_2, \ldots$) and split $\mathbb{N}$ into consecutive intervals of lengths $N_1$, $N_2$, etc. The function $f$ will cyclically map the interval of length $N_i$ to the elements of $B$ in the interval of length $n_i$ of the initial partition (we may assume for simplicity that $N_i$ is a multiple of the number of the elements of $B$ in the interval of length $n_i$).

Let $X$ be a negligible set. Let us show that its preimage $f^{-1}(X)$ is negligible. The fraction of $X$ in an interval of length $n_i$ converges to $0$ as $i\to\infty$ (we assume here that $n_i$ is greater than the sum of the lengths of the previous intervals, so deleting the previous intervals could increase the density at most by factor $2$, and the density of $X$ in the initial segment converges to $0$). On the other hand, the fraction of $B$-elements in the same interval of length $n_i$ does not converge to $0$ (exceeds $\eps$). Therefore, the fraction of $X$-elements among $B$-elements (in the same interval) converges to $0$. This fraction equals the density of $f^{-1}(X)$ in the interval of length $N_i$ of the second partition (in the preimage space). If the lengths $N_1, N_2, \ldots$ grow fast enough, then the fraction of  $f^{-1}(X)$ in the ``aligned'' initial segments (i.e., the initial segments that end on the boundaries between intervals of lengths $N_1,N_2,\ldots$) also converges to $0$.

However, we should also care about non-aligned initial segments. On the interval of length $N_i$ the mapping $f$ is periodic (see the construction above), and one period enumerates all the elements of $B$ in the $n_i$-interval. So the density of $f^{-1}(X)$ in each period is equal to the density of $X$ among the elelemts of $B$ (inside the $n_i$-interval), and this is OK. The problem is that while the density of $f^{-1}(X)$ in a single cycle inside $N_i$-interval is OK, this density may vary substantially in some parts of this cycle. But if $n_i$ is small enough compared to $N_1 + N_2 + \ldots + N_{i-1}$ (if $n_i = o(N_1 + N_2 + \ldots + N_{i-1})$, to be precise), then the weight of this part of the period in the density for the entire initial segment converges to $0$ (note that the length of the period inside $N_i$ is the number of $B$-elements in $n_i$-interval and therefore is bounded by $n_i$), so the last non-full period can be ignored.

This is the proof scheme. Let us list the requirements we used: 
\begin{itemize}
\item $n_1+\ldots+n_{i-1}=o(n_i)$, to get the lower bound for the density of $X$ in $n_i$-interval;
\item $N_1+\ldots+N_{i-1}=o(N_i)$, to estimate the density of $f^{-1}(X)$ in the initial segment that contains intervals of length $N_1,\ldots,N_{i}$;
\item $n_i=o(N_1+\ldots+N_{i-1})$, to deal with the non-aligned initial segments and the non-full periods. 
\end{itemize}
It is easy to see that these requirement can be fulfilled (all at the same time); first we choose $n_i$ satisfying the first requirement; then we choose $N_i$ that are large enough and grow fast enough. This finishes the proof of the Claim.
\end{proof}

This statement (used twice) provides two uniform functions whose images are contained in $B$ and its complement. Let us combine them, using the first function on the set of even numbers and the second one on the set of odd numbers. The combined function is also uniform. Indeed, for every negligible set $X$ its preimage is the union of two sets (the preimages for the two parts). One of them is negligible among the even numbers, while the other one is negligible among the odd numbers. Thus the combined preimage of $X$ is also negligible. Lemma~\ref{2} is proven.
\end{proof}
As we have seen, this finishes the proof of our theorem.
\end{proof}

It remains to note that the argument above can be effectivized in a straightforward way. If the given sets (the set $A$ in Lemma~\ref{1}, and the set $B$ that is not negligible and has non-negligible complement in Lemma~\ref{2}) are computable (decidable), the reduction functions are also computable for trivial reasons. (For Lemma~\ref{2} we need to know the value of $\eps$ to construct the required function, but we may choose and fix some rational $\eps$.) This remark finishes the proof of our main result : if $A \subset \mathbb{N}$ is a computable set, and $B$ is a computable set that is not negligible and has non-negligible complement,  then $A \leq_{gm} B$.

\end{document}